\documentclass{amsart}
\usepackage{graphicx}
\usepackage{amssymb}
\usepackage{amsfonts}
\usepackage{epsfig}
\swapnumbers
\newtheorem{theorem}{Theorem}[section]
\newtheorem{lemma}[theorem]{Lemma}

\newtheorem{corollary}[theorem]{Corollary}

\theoremstyle{definition}

\numberwithin{equation}{section}
 \theoremstyle{plain}    
 
 \numberwithin{equation}{section} %% Comment out for sequentially-numbered
 \numberwithin{figure}{section} %% Comment out for sequentially-numbered
 \theoremstyle{plain}    
 \theoremstyle{plain}    
 \theoremstyle{remark}    
 \newtheorem*{acknowledgement*}{Acknowledgement} 
\newcommand{\cF}{{\mathcal F}}

\newcommand{\cP}{{\mathcal P}}

\newcommand{\Om}{{\Omega}}
\newcommand{\om}{{\omega}}
\newcommand{\ve}{{\varepsilon}}
\newcommand{\del}{{\delta}}
\newcommand{\Del}{{\Delta}}
\newcommand{\gam}{{\gamma}}
\newcommand{\Gam}{{\Gamma}}
\newcommand{\vf}{{\varphi}}

\newcommand{\Sig}{{\Sigma}}
\newcommand{\sig}{{\sigma}}
\newcommand{\al}{{\alpha}}
\newcommand{\be}{{\beta}}
\newcommand{\ka}{{\kappa}}
\newcommand{\la}{{\lambda}}

\newcommand{\bbN}{{\mathbb N}}

\newcommand{\bbR}{{\mathbb R}}

\newcommand{\bbI}{{\mathbb I}}
\begin{document}
\title[]{Functional Erd\H os-R\' enyi Law of Large Numbers for Nonconventional
Sums under Weak Dependence.}%
 \vskip 0.1cm 
 \author{ Yuri Kifer\\
\vskip 0.1cm
Institute of Mathematics\\
Hebrew University\\
Jerusalem, Israel}%
\address{
 Institute of Mathematics, The Hebrew University, Jerusalem 91904, Israel}%
\email{ kifer@math.huji.ac.il}%

\thanks{}
\subjclass[2000]{Primary: 60F15 Secondary: 60F10, 60F17, 37D20, 37D35}%
\keywords{laws of large numbers, large deviations, nonconventional sums, 
hyperbolic  diffeomorphisms and flows, Markov processes.}%

\date{\today}
\begin{abstract}\noindent
We obtain a functional Erd\H os--R\' enyi law of large numbers for 
"nonconventional" sums of the form $\Sig_n=\sum_{m=1}^n
F(X_m,X_{2m},...,X_{\ell m})$ where $X_1,X_2,...$ is a sequence of exponentially
 fast $\psi$-mixing random vectors and $F$ is a Borel vector function extending
 in several directions \cite{Ki5} where only i.i.d. random variables 
 $X_1,X_2,...$ were considered.
\end{abstract}
%\footnotetext[1]{} 
\maketitle
\markboth{Y.Kifer}{Nonconventional sums} 
\renewcommand{\theequation}{\arabic{section}.\arabic{equation}}
\pagenumbering{arabic}

\section{Introduction}\label{sec1}\setcounter{equation}{0}

Let $X_1,X_2,...$ be a sequence of independent identically distributed
(i.i.d.) random variables such that $EX_1=0$ and the moment generating
function $\phi(t)=Ee^{tX_1}$ exists. Denote by $I$ the Legendre transform 
of $\ln \phi$ and set $\Sig_n=\sum_{m=1}^nX_m$ for $n\geq 1$ and $\Sig_0=0$. 
The Erd\" os-R\' enyi law of large
 numbers from \cite{ER} says that with probability one
\begin{equation}\label{1.1}
I(\al)\lim_{n\to\infty}\max_{0\leq m\leq n-[\frac {\ln n}{I(\al)}]}\frac
{\Sig_{m+[\frac {\ln n}{I(\al)}]}-\Sig_m}{\ln n}=\al
\end{equation}
for all $\al>0$ such that $I(\al)<\infty$.

The nonconventional limit theorems initiated in \cite{Ki4} and partially 
motivated by nonconventional ergodic theorems (with the name coming from 
\cite{Fu}) study asymptotic behaviors of sums of the form
\begin{equation}\label{1.2}
\Sig_n=\sum_{m=1}^nF(X_m,X_{2m},...,X_{\ell m})
\end{equation}
(and more general ones) where $F$ is a vector function satisfying certain 
conditions. The main features of such sums are nonstationarity and unboundedly
long (and strong) dependence of their summands. In \cite{Ki5} we established
(\ref{1.1}) for sums (\ref{1.2}) where $F$ is a bounded Borel function and
$X_1,X_2,...$ are independent identically distributed random variables. One of
the main reasons for the independence assumption in \cite{Ki5} was the use of
large deviations for nonconventional sums (\ref{1.2}) which was established in 
\cite{KV2} only for sums (\ref{1.2}) with i.i.d. random variables $X_1,X_2,...$.

In this paper we modify our method so that only the standard (conventional)
large deviations are used for sums of the form
\begin{equation}\label{1.3}
T_n=\sum_{m=1}^nF(X_m^{(1)},X_{2m}^{(2)},...,X^{(\ell)}_{\ell m})
\end{equation}
where $\{ X^{(i)}_m,\, m\geq 1\}$, $i=1,2,...,\ell$ are independent copies
of the sequence $\{ X_m,\, m\geq 1\}$. Now, when $X_1,X_2,...$ is a stationary
weakly dependent sequence then the latter sum consists of stationary weakly
dependent summands with similar properties which allows applications to
Markov chains satisfying the Doeblin condition and to some dynamical systems
such as Axiom A diffeomorphisms, expanding transformations and topologically
mixing subshifts of finite type (see \cite{Bow}). We assume exponentially fast
$\psi$-mixing of the sequence $X_1,X_2,...$ which still leads to long and 
strongly dependent summands of $\Sig_n$ but once we justify a transition to the
sums $T_n$ we arrive at exponentially fast $\psi$-mixing summands there.
Observe that the Erd\H os-R\' enyi law for conventional ($\ell=1$) sums of
exponentially fast $\psi$-mixing random variables was obtained in \cite{DK}.

In fact, we derive a functional form of the Erd\H os-R\' enyi
law for nonconventional sums (\ref{1.2}) which was first introduced for
(conventional) sums of i.i.d. random vectors in \cite{Bor} and it was never
considered before beyond this setup. This is a more
general result and as a corollary we derive from it the standard form of
the Erd\H os-R\' enyi law for nonconventional sums. Moreover, unlike the
original form of this law its functional form allows to consider a 
multidimensional version where $X_1,X_2,...$ are random vectors and $F$ is a
vector function.

The structure of this paper is as follows. In Section \ref{sec2} we describe
precisely our setup and results. In Section \ref{sec3} we exhibit a lemma
which is a version of Lemma 3.1 from \cite{Ki3} and which plays a crucial 
role here. In Sections \ref{sec4} and \ref{sec5} we derive the corresponding
upper and lower bounds which yield the functional form of the Erd\H os-R\' enyi 
law for nonconventional sums. After that we show how this implies the standard
form of this law. In Appendix we describe applications to Markov chains and
dynamical systems and then discuss some properties of rate functions of large
deviations which are relevant to our proofs but hard to find in most of the 
books on large deviations.

\section{Preliminaries and main results}\label{sec2}\setcounter{equation}{0}

Let $X_1,X_2,...$ be a $\wp$-dimensional stationary vector stochastic
process on a probability space $(\Om,\cF,P)$ and let $F:\bbR^{\wp\ell}\to
\bbR^d$ be a bounded Borel vector function on $\bbR^{\wp\ell}$.  
 Our setup includes also a sequence
 $\cF_{m,n}\subset\cF,\,-\infty\le m\le n\leq\infty$ of $\sig$-algebras such
 that $\cF_{m,n}\subset\cF_{m_1,n_1}$ whenever $m_1\le m$ and $n_1\ge n$ which 
 satisfies an exponentially fast $\psi$-mixing condition (see, for instance,
 \cite{Bra}),
 \begin{eqnarray}\label{2.1}
 &\psi(n)=\sup\big\{\big\vert\frac {P(A\cap B)}{P(A)P(B)}-1\big\vert :\, A\in
 \cF_{-\infty,k},\, B\in\cF_{k+n,\infty},\\
 & P(A)P(B)\ne 0\big\}\leq\ka_1^{-1}e^{-\ka_1n}\nonumber
 \end{eqnarray}
 for some $\ka_1>0$ and all $k,n\geq 0$. 
 
 We assume also the centering condition
 \begin{equation}\label{2.2}
 \bar F=\int F(x_1,x_2,...,x_\ell)d\mu(x_1)...d\mu(x_\ell)=0,
 \end{equation}
 where $\mu$ is the distribution of $X_1$, which is not actually a restriction
  since we always can take $F-\bar F$ in place of $F$. In addition, we assume
  that either $X_n$ is $\cF_{n-m,n+m}$-measurable for some $m\in\bbN$ 
  independent of $n$ and then $F$ is supposed to be only Borel measurable
  and bounded or $F$ is supposed to be bounded and H\" older continuous 
  \begin{equation}\label{2.3}
  \| F\|_\infty=D<\infty,\,\, |F(x_1,...,x_\ell)-F(y_1,...,y_\ell)|\leq
  \ka_2^{-1}\sum_{i=1}^\ell|x_i-y_i|^{\ka_2},\,\ka_2>0
  \end{equation}
  and then we need only the following approximation property
  \begin{equation}\label{2.3+}
  E|X_n-X_{n,m}|^{\ka_2}\leq\ka_3^{-1}e^{-\ka_3m},\,\, X_{n,m}=E(X_n|
  \cF_{n-m,n+m})
  \end{equation}
  for all $n,m\in\bbN$ and some $\ka_3>0$ independent of $n$ and $m$.
  
  Define two sums
\[
\Sig_n=\sum_{1\leq k\leq n}F(X_k,X_{2k},...,X_{\ell k}),\,\,\Sig_0=0
\]
and
\[
T_n=\sum_{1\leq k\leq n} F(X_k^{(1)},X_{2k}^{(2)},...,X_{\ell k}^{(\ell)}),
\,\,T_0=0 
\]
where $\{ X_k^{(i)},\, k\geq 1\},\, i=1,2,...,\ell$ are independent copies
(in the sense of distributions) of the stationary process $\{ X_k,\, k\geq 1\}$.
Assume that for any piece-wise constant map $\gam:\,[0,1]\to\bbR^d$ the
limit 
\begin{equation}\label{2.4}
\lim_{n\to\infty}\frac 1n\ln E\exp\big(n\int_0^1(\gam_u,F(X_{[un]}^{(1)},
X_{2[un]}^{(2)},...,X_{\ell[un]}^{(\ell)}))du\big)=\int_0^1\Pi(\gam_u)du
\end{equation}
exists, $\Pi(\al),\,\al\in\bbR^d$ is a convex twice differentiable function
such that $\nabla_\al\Pi(\al)|_{\al=0}=0$ and the Hessian matrix 
$\nabla^2_\al\Pi(\al)|_{\al=0}$ is positively definite (where $(\cdot,\cdot)$
denotes the inner product). 

Let
\begin{equation}\label{2.5}
I(\be)=\sup_\al((\al,\be)-\Pi(\al))
\end{equation}
and for any $\gam:[0,1]\to\bbR^d$ from the space $C([0,1],\bbR^d)$ of 
continuous curves on $\bbR^d$ set
\begin{equation}\label{2.6}
S(\gam)=\int_0^1I(\dot\gam_u)du
\end{equation}
if $\gam$ is absolutely continuous and $S(\gam)=\infty$, for otherwise.
It follows from the existence and properties of the limit (\ref{2.4}) (see,
for instance, Section 7.4 in \cite{FW}) that $n^{-1}T_n$ satisfies large 
deviations estimates in the form that for any $a,\del,\la>0$ and every
$\gam\in C([0,1],\bbR^d),\,\gam_0=0$ there exists $n_0>0$ such that for
$n\geq n_0$,
\begin{eqnarray}\label{2.7}
& P\{\rho(n^{-1}T_n,\,\gam)<\del\}\geq\exp(-n(S(\gam)+\la))\,\,\mbox{and}\\
& P\{\rho(n^{-1}T_n,\,\Phi(a))\geq\del\}\leq\exp(-n(a-\la))\nonumber
\end{eqnarray}
where $\rho(\gam,\eta)=\sup_{u\in[0,1]}|\gam_u-\eta_u|$ and $\Phi(a)=
\{\gam\in C([0,1],\bbR^d):\,\gam_0=0,\, S(\gam)\leq a\}$.

Since $S$ is a lower semi-continuous functional then each $\Phi(a),\, a<\infty$
 is a closed set and, moreover, it is compact for any finite $a$. Indeed, 
 $|\Pi(\al)|\leq D|\al|$ by (\ref{2.3}) which implies
 by (\ref{2.5}) that $I_t(\be)=\infty$ provided $|\be|>D$ (take 
 $\al=a\be/|\be|$ in (\ref{2.5}) and let $a\to\infty$). Hence, $|\dot\gam_s|
 \leq D$ for Lebesgue almost all $s\in[0,1]$ if $\gam\in\Phi_t(a)$, and
 so the latter set is bounded and equicontinuous which by the Arzel\` a-Ascoli
 theorem implies its compactness.

For each $c>0,\, u\in[0,1]$ and integers $m\geq 0$ and $n\geq 2$ set
\[
V^c_{m,n}(u)=\frac {\Sig_{m+b_c(n,u)}-\Sig_m}{b_c(n)}
\]
where $b_c(n,u)=[ub_c(n)]$ and $b_c(n)=[c\ln n]$. Introduce also the set of
random curves $W^c_n=\{ V^c_{j,n}:\, 0\leq j\leq n-b_c(n)\}$. The following 
is the main result of this paper.

\begin{theorem}\label{thm2.1}
Assume that the conditions (\ref{2.1})--(\ref{2.4}) hold true. Then, for any 
$c>0$ with probability one
\begin{equation}\label{2.8}
\lim_{n\to\infty}H(W^c_n,\,\Phi(1/c))=0
\end{equation}
where $H(\Gam_1,\Gam_2)=\inf\{\del>0:\,\Gam_1\subset\Gam_2^\del,\,
\Gam_2\subset\Gam_1^\del\}$ is the Hausdorff distance between sets of curves
with respect to the uniform metric $\rho$ and $\Gam^\del=\{\gam:\,\rho(\gam,
\Gam)<\del\}$.
\end{theorem}
\begin{corollary}\label{cor2.2} Let $d=1$. Then for $c=\frac 1{I(\be)}$ with
probability one
\begin{equation}\label{2.9}
\lim_{n\to\infty}\max_{0\leq k\leq n-b_c(n)}V^c_{k,n}(1)=\be
\end{equation}
provided $0<\be<\be_0=\sup\{\be:\, I(\be)<\infty\}$.
\end{corollary}

Observe that $(X_k^{(1)},X_{2k}^{(2)},...,X^{(\ell)}_{\ell k}),\, k\geq 1$ is
an $\ell\wp$-dimensional stationary process with properties similar to the
ones of the process $X_k,\, k\geq 1$, and so unlike $\Sig_n$ the sum $T_n$
requires only "conventional" treatment. Our main goal here will be to show how
to replace in our proofs the handling of the sums $\Sig_n$ by the sums $T_n$.
We will mainly discuss the proof for the case where $F$ and $X_n,\, n\geq 1$
satisfy the conditions (\ref{2.3}) and (\ref{2.3+}) since the case when $X_n$
 is $\cF_{n-m,n+m}$-measurable and $F$ is only a bounded Borel function is
 established by an obvious simplification of the proof just by eliminating
 the steps connected to approximations of $X_n$ by corresponding conditional
 expectations $E(X_n|\cF_{n-m,n+m})$.
 
 Our method goes through also for more general sums $\Sig_n=\sum_{1\leq m\leq n}
 F(X_{q_1(m)},X_{q_2(m)},...,X_{q_\ell(m)})$ where $q_i(m)=im$ for $i\leq k\leq
 \ell$ and $q_j(m)$ for $j=k+1,...,\ell$ being nonlinear indexes as in 
 \cite{KV1}. For instance, we may take $q_j(m)=m^j$ for $j>k$. In this 
 situation, it turns out that we can replace such sums $\Sig_n$ by the sums
 $T_n=\sum_{1\leq m\leq n} F(X^{(1)}_{q_1(m)},X^{(2)}_{q_2(m)},...,
 X^{(\ell)}_{q_\ell(m)})$ where, again, $\{X^{(i)}_m,\, m\geq 1\}$, $i=1,...,
 \ell$ are independent processes, for $i=1,...,k$ they are copies of $X_m,\, 
 m\geq 1$ while all $X_m^{(i)},\, m\geq 1,i>k$ are i.i.d. and have the same 
 distribution as $X_1$. Then, again, $(X_{q_1(m)}^{(1)},X_{q_2(m)}^{(2)},...,
 X^{(\ell)}_{q_\ell(m)}),\, m\geq 1$ is an $\ell\wp$-dimensional stationary 
 process with properties similar to the ones of the process $X_m,\, m\geq 1$
 and we can deal with such sums $T_n$ in the same way as in this paper.
 
 Our conditions are satisfied when, for instance, $X_n,\, n\geq 1$ is a 
 $\wp$-dimensional Markov chain with transition
 probabilities $P(x,\Gam)$ satisfying the strong Doeblin type condition
 \[
 C^{-1}\leq\frac {P(x,dy)}{d\nu(y)}\leq C
 \]
 for some probability measure $\nu$ and a constant $C>0$ independent of $x$ and 
  $y$. Then $(X_n^{(1)},X_{2n}^{(2)},...,X_{\ell n}^{(\ell)})$ is an 
  $\ell\wp$-dimensional Markov chain satisfying similarly to $X_n$ both
  exponentially fast $\psi$-mixing and the necessary large deviations 
  estimates (see \cite{Bra}, \cite{FW} and \cite{Ki1}). Here we can take,
  for instance, the $\sig$-algebras $\cF_{m,n}$ generated by $X_m,X_{m+1},
  ...,X_m$ and then $F$ is supposed to be only bounded and Borel.
  
  On the dynamical systems side our conditions are satisfied, for instance,
  when $X_n=g\circ f^n$ where $g$ is a H\" older continuous function and
  $f$ is an Axiom A diffeomorphism on a hyperbolic set, expanding 
  transformation or a mixing subshift of finite type (see \cite{Bow}). 
  In this case
  $f:M\to M$ and $X_n(\om),\,\om\in M$ is a stationary sequence on the
  probability space $(M,\cF,P)$ where $M$ is the corresponding phase space,
  $\cF$ is a Borel $\sig$-algebra and $P$ is a Gibbs measure constructed by
  a H\" older continuous function. The function $F$ here should satisfy 
  (\ref{2.3}) and the $\sig$-algebras $\cF_{m,n}$ are generated by cylinder
  sets in the subshift case or by corresponding Markov partitions in the
  Axiom A and expanding case. The exponentially fast $\psi$-mixing for these
  transformations is obtained in \cite{Bow} and the required large deviations
   results can be found in \cite{Ki1} and \cite{Ki2} and the product system
   $(f,f^2,...,f^\ell)$, which plays the role of $\ell$ independent copies
    in sums $T_n$, has similar properties to the dynamical system $f^n$ itself.

  \section{Basic estimates}\label{sec3}\setcounter{equation}{0}

We start with the following result which is a corollary of Lemma 3.1
from \cite{Ki3}.
\begin{lemma}\label{lem3.1}
Let $Y_i$ be $\wp_i$-dimensional random vectors with a distribution $\mu_i,\,
i=1,...,k$ defined on the same probability space $(\Om,\cF,P)$ and such
that $Y_i$ is $\cF_{m_i,n_i}$-measurable where $n_{i-1}<m_i\leq n_i<m_{i+1},\,
i=1,...,k$, $n_0=-\infty,\, m_{k+1}=\infty$ and $\sig$-algebras $\cF_{n,m}$
satisfy the condition (\ref{2.1}). Then, for any bounded Borel function
$h=h(x_1,...,x_k)$ on $\bbR^{\wp_1+\wp_2+\cdots +\wp_k}$,
\begin{eqnarray}\label{3.1}
&|Eh(Y_1,Y_2,...,Y_k)-\int h(x_1,x_2,...,x_k)d\mu_1(x_1)d\mu_2(x_2)...
d\mu_k(x_k)|\\
&\leq\ka_1^{-1}\| h\|_\infty\sum_{i=2}^ke^{-\ka_1(m_i-n_{i-1})}\nonumber
\end{eqnarray}
where $\|\cdot\|$ is the $L^\infty$ norm. In other words, if $Y_1^{(1)},
Y_2^{(2)},...Y_k^{(k)}$ are independent copies of $Y_1,Y_2,...,Y_k$, 
respectively, then
\begin{equation}\label{3.2}
|Eh(Y_1,Y_2,...,Y_k)-Eh(Y_1^{(1)},Y_2^{(2)},...,Y_k^{(k)})|\leq\ka_1^{-1}
\| h\|_\infty\sum_{i=2}^ke^{-\ka_1(m_i-n_{i-1})}.
\end{equation}
Taking $h=\bbI_\Gam$ for a Borel set $\Gam\subset\bbR^{\wp_1+\wp_2+\cdots 
+\wp_k}$ (where $\bbI_\Gam (x)=1$ if $x\in\Gam$ and $\bbI_\Gam (x)=0$ for 
otherwise) it follows that
\begin{equation}\label{3.3}
|P\{ (Y_1,Y_2,...,Y_k)\in\Gam\}-P\{ (Y_1^{(1)},Y_2^{(2)},...,Y_k^{(k)})\in
\Gam\}|\leq\ka_1^{-1}\sum_{i=2}^ke^{-\ka_1(m_i-n_{i-1})}.
\end{equation}
\end{lemma}
\begin{proof} If $k=2$ then Lemma 3.1 from \cite{Ki3} gives that
\[
|E(h(Y_1,Y_2)|\cF_{m_1,n_1})-g(Y_1)|\leq\psi(m_2-n_1)\| h\|_\infty
\]
where $g(y)=Eh(y,Y_2)=\int h(y,z)d\mu_2(z)$. Taking the expectation we obtain 
(\ref{3.1}) for $k=2$. Now let (\ref{3.1}) holds true for all $k\leq j-1$ and 
any bounded Borel function of the corresponding number of arguments. In order
to derive (\ref{3.1}) for $k=j$ we consider $(Y_1,Y_2,...,Y_{j-1})$ as one
random vector and $Y_j$ as another . Then we obtain from Lemma 3.1 of 
\cite{Ki3} that
\[
|E(h(Y_1,Y_2,...,Y_j)|\cF_{m_1,n_{j-1}})-g(Y_1,Y_2,...,Y_{j-1})|\leq\psi
(m_j-n_{j-1})\| h\|_\infty
\]
where $g(y_1,y_2,...,y_{j-1})=Eh(y_1,y_2,...,y_{j-1},Y_j)$. Now, taking the
expectation and applying the induction hypothesis to $g$ we complete the
proof.
\end{proof}

In the case when $X_n$ is $\cF_{n-m,n+m}$-measurable for all $n$ and a fixed
$m$ then we will be able to use Lemma \ref{lem3.1} directly which will enable
us to replace summands $F(X_n,X_{2n},...,X_{\ell n})$ by $F(X_n^{(1)},
X_{2n}^{(2)},...,X_{\ell n}^{(\ell)})$. On the other hand, under (\ref{2.3})
and (\ref{2.3+}) we will have, first, to replace the original random vectors
$X_m$ by their approximations $X_{m,k}=E(X_m|\cF_{m-k,m+k})$ and then using
(\ref{2.3}) to estimate the error.

Namely, for $\bar k=(k_1,k_2,...,k_n)$ set
\[
R_{n,\bar k}=\sum_{1\leq j\leq n}F(X_{j,k_j},X_{2j,k_j},...,X_{\ell j,k_j}).
\]
Let $j(u)>0,\, u\in[0,1]$ be a non decreasing integer valued function. We will
use that by (\ref{2.3}) and (\ref{2.3+}),
\begin{eqnarray}\label{3.4}
&E\sup_{u\in[0,1]}|\Sig_{m+j(u)}-\Sig_m-(R_{m+j(u),\bar k}-R_{m,\bar k})|\\
&\leq\sum_{m+1\leq i\leq m+j(1)}E|F(X_i,X_{2i},...,X_{\ell i})
-F(X_{i,k_i},X_{2i,k_i},...,X_{\ell i,k_i})|\nonumber\\
&\leq\ka_2^{-1}\ka_3^{-1}\sum_{m+1\leq i\leq m+j(1)}e^{-\ka_3k_i}.
\nonumber\end{eqnarray}

We observe that Lemma \ref{lem3.1} applied to the summands of the form 
$F(X_{j,k_j},...,X_{\ell j,k_j})$ does not yield yet the summands of the
form $F(X_j^{(1)},...,X_{\ell j}^{(\ell)})$ but only the summands
$F(X_{j,k_j}^{(1)},...,X_{\ell j,k_j}^{(\ell)})$ where $X^{(i)}_{ij,k_j},\,
i=1,...,\ell$ are independent and have the same distributions as
$X_{ij,k_j},\, i=1,...,\ell$, respectively. Thus an additional argument 
together with another use of (\ref{2.3}) and (\ref{2.3+}) will be needed.

\section{The upper bound}\label{sec4}\setcounter{equation}{0}

We will show first that with probability one,
\begin{equation}\label{4.1}
\lim_{n\to\infty}\max_{0\leq j\leq n-b_c(n)}\rho(V^c_{j,n},\Phi(1/c))=0.
\end{equation}
This assertion means that with probability one all limit points as 
$n\to\infty$  of curves from $W_n^c$ belong to the compact set $\Phi(1/c)$.

Set $R_n=R_{n,\bar k}$ where $\bar k=(k_j,\, 1\leq j\leq n)$ with $k_j=
[j/3]$ and
\[
\hat V^c_{m,n}(u)=\frac {R_{m+b_c(n,u)}-R_m}{b_c(n)}.
\]
Then by (\ref{3.4}) and the Chebyshev inequality
\begin{eqnarray}\label{4.2}
&P\{\rho(V^c_{m,n},\,\hat V^c_{m,n})\geq\ve\}\leq\ve^{-1}b_c^{-1}(n)\ka_2^{-1}
\ka_3^{-1}\sum_{m+1\leq j<\infty}e^{-\ka_3[j/3]}\\
&\leq \ve^{-1}b_c^{-1}\ka_2^{-1}\ka_3^{-1}e^{-\ka_3(m-2)/3}(1-e^{-\frac 
13\ka_3})^{-1}.
\nonumber\end{eqnarray}

Observe that
\begin{equation}\label{4.3}
|\Sig_{m+b_c(n,u)}-\Sig_m-(\Sig_{m+k+b_c(n,u)}-\Sig_{m+k})|\leq 2kD\quad
\mbox{a.s.}
\end{equation}
where, recall, $D=\| F\|_\infty$. Hence, a.s.,
\begin{equation}\label{4.4}
0\leq\max_{0\leq j\leq n-b_c(n)}\rho(V^c_{j,n},\Phi(1/c))-\max_{\ve b_c(n)\leq
 j\leq n-b_c(n)}\rho(V^c_{j,n},\Phi(1/c))\leq 2\ve D
 \end{equation}
 and, similarly,
 \begin{equation}\label{4.4+}
0\leq\max_{0\leq j\leq n-b_c(n)}\rho(\hat V^c_{j,n},\Phi(1/c))-\max_{\ve b_c(n)
\leq j\leq n-b_c(n)}\rho(\hat V^c_{j,n},\Phi(1/c))\leq 2\ve D.
 \end{equation}
 By (\ref{4.2}) we also have that
 \begin{eqnarray}\label{4.5}
 &P\{\max_{\ve b_c(n)\leq j\leq n-b_c(n)}\rho(V^c_{j,n},\hat V^c_{j,n})
 \geq\ve\}\leq\sum_{\ve b_c(n)\leq j\leq n-b_c(n)}\\
 &P\{\rho(V^c_{j,n},\hat V^c_{j,n}) \geq\ve\}\leq (\ve b_c(n)\ka_2\ka_3
 (1-e^{-\ka_3/3})^2)^{-1}e^{2\ka_3/3} n^{-\ka_3\ve c/3}.\nonumber
 \end{eqnarray}
 
 Next, observe that the family of compacts $\{\Phi(a),\, a>0\}$ is upper
 semi continuous, i.e. for any $a>0$ and $\ve>0$ there exists $\del>0$ such
 that $\Phi(a+\del)\subset\Phi(a)^\ve$ where, as before, $G^\ve$ denotes
 the open $\ve$-neighborhood of a set $G$. Indeed, if $G_\ve=\Phi(a+\del)
 \setminus\Phi(a)^\ve\ne\emptyset$ for some $a,\ve>0$ and all $\del$ then
 the decreasing with $\del\downarrow 0$ compact sets $G_\del$ must have a
 common point $\gam_0\in\Phi(a)$ contradicting the fact that 
 $\gam_0\not\in\Phi(a)^\ve$. Now, choosing such $\del$ for $\ve$ and $a=1/c$
 we obtain that
 \begin{equation}\label{4.6}
 P\{\rho(\hat V^c_{j,n},\Phi(1/c))>3\ve\}\leq 
 P\{\rho(\hat V^c_{j,n},\Phi(1/c+\del))>2\ve\}.
 \end{equation}
 
 For each vector 
 \[
 \bar x(m)=(x_1^{(1)},...,x_m^{(1)};x_1^{(2)},...,x_m^{(2)};...;x_1^{(\ell)},
 ...,x_m^{(\ell)})\in\bbR^{m\ell\wp}
 \]
 define the curve $\gam(\bar x(m))$ in $\bbR^d$ by
 \[
 \gam_u(\bar x(m))=m^{-1}\sum_{1\leq j\leq [um]}F(x_j^{(1)},x_j^{(2)},...,
 x_j^{(\ell)}),\, u\in[0,1]
 \]
 where the sum over the empty set is considered to be zero.
 Introduce the Borel set
 \[
 \Gam=\{\bar x(b_c(n))\in\bbR^{b_c(n)\ell\wp}:\,\rho(\gam(\bar x(b_c(n)),\,
 \Phi(\frac 1c+\del))> 2\ve\}.
 \]
 
 Let the pairs $\{ (X^{(i)}_{ij},\, X^{(i)}_{ij,[j/3]}),\, j=1,2,...\},$
 $i=1,...,\ell$ be independent copies (in the sense of joint distributions) of
 pairs of processes $\{(X_{ij},\, X_{ij,[j/3]}),\, j=1,2,...\}$, 
 $i=1,...,\ell$, respectively, which can be constructed on a product space.
 We identify the processes $\{ X^{(i)}_{ij},\, j=1,2,...\},\, i=1,...,\ell$
 with the processes having the same notation in Section \ref{sec2} since they
 have the same joint distributions which is what only matters here. Set 
 \[
 \hat T_n=\sum_{j=1}^nF(X^{(1)}_{j,[j/3]},X^{(2)}_{2j,[j/3]},...,
 X^{(\ell)}_{\ell j,[j/3]})
 \]
 and
 \[
 \hat U^c_{m,n}(u)=\frac {\hat T_{m+b_c(n,u)}-\hat T_m}{b_c(n)},\, u\in[0,1].
 \]
 
 Next, observe that
 \[
 \{ (X_{j,[j/3]},X_{2j,[j/3]},...,X_{\ell j,[j/3]})_{m+1\leq j\leq m+b_c(n)}
 \in\Gam\}=\{\rho(\hat V^c_{m,n},\Phi(\frac 1c+\del))>2\ve\}
 \]
 and 
 \[
 \{ (X^{(1)}_{j,[j/3]},X^{(2)}_{2j,[j/3]},...,X^{(\ell)}_{\ell j,
 [j/3]})_{m+1\leq j\leq m+b_c(n)}
 \in\Gam\}=\{\rho(\hat U^c_{m,n},\Phi(\frac 1c+\del))>2\ve\}.
 \]
 Taking in Lemma \ref{lem3.1},
 \[
 Y_i=(X_{ij,[j/3]},\, j=m+1,m+2,...,m+b_c(n)),\, i=1,2,...,\ell
 \]
 and observing that 
 \begin{equation}\label{4.7}
 i(m+1)-\frac 13(m+1)-(i-1)(m+b_c(n))-\frac 13(m+b_c(n))\geq\frac 16m,
 \end{equation}
 provided that $m\geq 6\ell b_c(n)$, we obtain from (\ref{3.3}) that for 
 $m\geq 6\ell b_c(n)$,
 \begin{equation}\label{4.8}
 |P\{\rho(\hat V^c_{m,n},\,\Phi(\frac 1c+\del))>2\ve\}-
 P\{\rho(\hat U^c_{m,n},\,\Phi(\frac 1c+\del))>2\ve\}|\leq\ka_1^{-1}\ell 
 e^{-\ka_1m/6}.
 \end{equation}
 
 Next, using the same notation for $T_n=\sum_{j=1}^nF(X^{(1)}_j,X^{(2)}_{2j},
 ...,X^{(\ell)}_{\ell j})$ as in Section \ref{sec2} with $X^{(1)}_j,
 X^{(2)}_{2j},...,X^{(\ell)}_{\ell j}$ introduced in this section we set
 \[
  U^c_{m,n}(u)=\frac { T_{m+b_c(n,u)}- T_m}{b_c(n)},\, u\in[0,1].
 \]
 Now, recall that each pair $(X^{(i)}_{ij},\, X^{(i)}_{ij,[j/3]})$ has the
 same joint distribution as $(X_{ij},\, X_{ij,[j/3]})$, and so by (\ref{2.3+}),
 \begin{equation}\label{4.8+}
 E|X^{(i)}_{ij}-X^{(i)}_{ij,[j/3]}|=E|X_{ij}-E(X_{ij}|\cF_{ij-[j/3],
 ij+[j/3]})|\leq\ka_3^{-1}e^{-\ka_3[j/3]}.
 \end{equation}
 Hence, similarly to (\ref{3.4}) and (\ref{4.2}) we conclude that
 \begin{equation}\label{4.8++}
 P\{\rho(U^c_{m,n},\hat U^c_{m,n})>\ve\}\leq (\ve b_c(n)\ka_2\ka_3
 (1-e^{-\ka_3/3}))^{-1}e^{-\ka_3(m-2)/3}.
 \end{equation}
 Now, repeating (\ref{4.4})--(\ref{4.5}) with $U^c$ in place of $V^c$ we 
 write also
 \begin{equation}\label{4.8b}
0\leq\max_{0\leq j\leq n-b_c(n)}\rho(U^c_{j,n},\Phi(1/c))-\max_{\ve b_c(n)\leq
 j\leq n-b_c(n)}\rho(U^c_{j,n},\Phi(1/c))\leq 2\ve D,
 \end{equation}
 \begin{equation}\label{4.8bb}
0\leq\max_{0\leq j\leq n-b_c(n)}\rho(\hat U^c_{j,n},\Phi(1/c))-\max_{\ve b_c(n)
\leq j\leq n-b_c(n)}\rho(\hat U^c_{j,n},\Phi(1/c))\leq 2\ve D,
 \end{equation}
 \begin{eqnarray}\label{4.8bbb}
 &P\{\max_{\ve b_c(n)\leq j\leq n-b_c(n)}\rho(U^c_{j,n},\hat U^c_{j,n})
 \geq\ve\}\leq\sum_{\ve b_c(n)\leq j\leq n-b_c(n)}\\
 &P\{\rho(U^c_{j,n},\hat U^c_{j,n}) \geq\ve\}\leq (\ve b_c(n)\ka_2\ka_3
 (1-e^{-\ka_3/3})^2)^{-1}e^{2\ka_3/3} n^{-\ka_3\ve c/3}.\nonumber
 \end{eqnarray}
 
 Taking into account stationarity of the sequence $F(X^{(1)}_m,X^{(2)}_{2m},
 ...,X^{(\ell)}_{\ell m}),\, m\geq 1$ we obtain from the upper large deviations 
 bound in (\ref{2.7}) that for any $\ve,\la>0$ there exists $n(\ve,\la)$ such
 that for all $n\geq n(\ve,\la)$,
 \begin{eqnarray}\label{4.9}
 &P\{\rho(U^c_{m,n},\Phi(\frac 1c+\del))\geq\ve\}=P\{\rho(b^{-1}_c(n)T_{b_c(n)},
 \Phi(\frac 1c+\del))\geq\ve\}\\
 &\leq e^{-(\frac 1c+\del-\la)b_c(n)}\leq e^{-\frac 1c}n^{-(1+c\sig)}
 \nonumber\end{eqnarray}
 where we choose $\la>0$ so small that $\sig=\del-\la>0$.
 
 Now, by (\ref{4.2}), (\ref{4.6}), (\ref{4.8}), (\ref{4.8+}) and (\ref{4.9})
 we obtain that for $m\geq 6\ell b_c(n)$,
 \begin{equation}\label{4.10}
 P\{\rho(V^c_{m,n},\Phi(1/c))>4\ve\}\leq d_1^{-1}(e^{-d_1m}+n^{-(1+c\sig)})
 \end{equation}
 for some $d_1=d_1(\ve)>0$ independent of $m$ and $n$ but dependent on
 $\ve>0$ which is fixed for now. Hence, we obtain from (\ref{4.2}), (\ref{4.4}),
 (\ref{4.4+}), (\ref{4.6}), (\ref{4.8b}), (\ref{4.8bb}) and (\ref{4.10}) that
 \begin{eqnarray}\label{4.11}
 &P\{\max_{0\leq m\leq n-b_c(n)}\rho(V^c_{m,n},\Phi(1/c))>8D+4\ve\}\\
 &\leq P\{\max_{\ve b_c(n)\leq m\leq n-b_c(n)}\rho(V^c_{m,n},\Phi(1/c))>4\ve\}
 \nonumber\\
 &\leq P\{\max_{\ve b_c(n)\leq m<6\ell b_c(n)}\rho(V^c_{m,n},\Phi(1/c))>4\ve\}
 \nonumber\\
 &+\sum_{6\ell b_c(n)\leq m\leq n-b_c(n)}P\{\rho(V^c_{m,n},\Phi(1/c))>4\ve\}
 \nonumber\\
 &\leq \sum_{\ve b_c(n)\leq m\leq 6\ell b_c(n)}P\{\rho(\hat V^c_{m,n},\Phi(1/c))
 >3\ve\}+d_2^{-1}n^{-d_2}\nonumber\\
 &\leq \sum_{\ve b_c(n)\leq m\leq 6\ell b_c(n)}P\{\rho(\hat V^c_{m,n},
 \Phi(1/c+\del))>2\ve\}+d_2^{-1}n^{-d_2}\nonumber
 \end{eqnarray}
 for some $d_2>0$ independent of $n$ but dependent on $\ve>0$.
 
 Next, we have to modify our approach for $\ve b_c(n)\leq m<6\ell b_c(n)$ in
 order to estimate the last sum in the right hand side of (\ref{4.11}). Fix
 an integer $M$ and define new random curves setting for $\frac {k-1}M\leq u
 <k/M$,
 \[
 \hat V^{c,M}_{m,n,k}(u)=M(\hat V^c_{m,n}(u)-\hat V^c_{m,n}(\frac {k-1}M)),\,
 k=1,...,M
 \]
 while $\hat V^{c,M}_{m,n,k}(u)=0$ if $u\not\in[\frac {k-1}M,\frac kM)$.
 We define also for $u\in[\frac {k-1}M,\frac kM)$,
 \[
 \hat U^{c,M}_{m,n,k}(u)=M(\hat U^c_{m,n}(u)-\hat U^c_{m,n}(\frac {k-1}M)),\,
 k=1,...,M
 \]
 with $\hat U^{c,M}_{m,n,k}(u)=0$ if $u\not\in[\frac {k-1}M,\frac kM)$. Observe
 that for $u\in[\frac {k-1}M,\frac kM)$,
 \[
 \hat V^{c,M}_{m,n,k}(u)=Mb_c^{-1}(n)\sum_{j=m+[(k-1)M^{-1}b_c(n)]+1}^{m+
 b_c(n,u)}F(X_{j,[j/3]},X_{2j,[j/3]},...,X_{\ell j,[j/3]})
 \]
 and
 \[
 \hat U^{c,M}_{m,n,k}(u)=Mb_c^{-1}(n)\sum_{j=m+[(k-1)M^{-1}b_c(n)]+1}^{m+
 b_c(n,u)}F(X^{(1)}_{j,[j/3]},X^{(2)}_{2j,[j/3]},...,X^{(\ell)}_{\ell j,
 [j/3]}).
 \]
 
 Taking in Lemma \ref{lem3.1},
 \[
 Y_i=(X_{ij,[j/3]},\, j=m+[(k-1)M^{-1}b_c(n)]+1,...,m+[kM^{-1}b_c(n)]),\,
 i=1,2,...,\ell
 \]
 and observing that 
 \begin{equation}\label{4.11+}
 (i-\frac 13)(m+(k-1)M^{-1}b_c(n))-(i-\frac 23)(m+kM^{-1}b_c(n))-\frac 13
 \geq (\frac \ve 3+\frac {k+1}{3M}-\frac \ell M)b_c(n),
 \end{equation}
 provided that $m\geq\ve b_c(n)$,
 we see that for $M=M(\ve)=6\ell([1/\ve]+1)$ the right hand 
 side of (\ref{4.11+}) is not less than $\frac 16\ve b_c(n)$. 
  Thus by (\ref{3.3}) for such $m$ in the same way as in
 (\ref{4.8}) we obtain that
 \begin{eqnarray}\label{4.11++}
 &|P\{\rho(\hat V^{c,M}_{m,n,k},\,\Phi(\frac 1c+\del))>2\ve\}-
 P\{\rho(\hat U^{c,M}_{m,n,k},\,\Phi(\frac 1c+\del))>2\ve\}|\\
 &\leq\ka_1^{-1}e^{\ka_1(1+\ve)/3}\ell e^{-\frac 16\ka_1
 \ve c\ln n}=\ka_1^{-1}e^{\ka_1(1+\ve)/3}\ell n^{-\ka_1\ve c/6}.\nonumber
 \end{eqnarray}
 Since there are no more than $6\ell c\ln n$ numbers $m$ for which we will 
 need this estimate it will suit our purposes.
 
 Next, define for $u\in[\frac {k-1}M,\frac kM)$ and $k=1,...,M$,
 \begin{eqnarray*}
  &U^{c,M}_{m,n,k}(u)=M(U^c_{m,n}(u)-U^c_{m,n}(\frac {k-1}M))\\
  &=Mb_c^{-1}(n)\sum_{j=m+[(k-1)M^{-1}b_c(n)]+1}^{m+
 b_c(n,u)}F(X^{(1)}_{j,},X^{(2)}_{2j},...,X^{(\ell)}_{\ell j}).
 \end{eqnarray*}
 Relying on (\ref{2.3}) and (\ref{4.8+}) we estimate $P\{\rho(U^{c,M}_{m,n,k},
 \hat U^{c,M}_{m,n,k})>\ve\}$ by the right hand side of (\ref{4.8++}) and
 using (\ref{4.11++}) we obtain
 \begin{equation}\label{4.11+++}
  P\{\rho(\hat V^{c,M}_{m,n,k},\,\Phi(\frac 1c+\del))>2\ve\}\leq
  P\{\rho(U^{c,M}_{m,n,k},\,\Phi(\frac 1c+\del))>2\ve\}+d^{-1}_3n^{-d_3}
  \end{equation}
  for some $d_3>0$ independent of $m,n,k$ but dependent on $\ve$.
  
  Next, using stationarity of the sequence $F(X^{(1)}_k,X^{(2)}_{2k},...,
  X_{\ell k}^{(\ell)}),\, k\geq 1$ we can compute the rate functional of
  large deviations for $U^{c,M}_{m,n,k}$ as $n\to\infty$ from 
  (\ref{2.4})--(\ref{2.6}) which will provide the upper bound similarly to
  (\ref{4.9}) in the form
  \begin{equation}\label{4.12}
 P\{\rho(U^{c,M}_{m,n,k},\Phi(\frac 1c+\del))\geq\ve\}
 \leq d_4^{-1}e^{-d_4\ln n}=d_4^{-1} n^{-d_4}
 \end{equation}
 for some $d_4>0$ depending on $\ve$ but not on $n$. Now observe that
 \[
  \hat V^{c}_{m,n}=\frac 1M\sum_{k=1}^M\hat V^{c,M}_{m,n,k}.
 \]
 Since $\Phi(\frac 1c+\del)$ is a convex set then
 \[
 \rho( \hat V^{c}_{m,n},\Phi(\frac 1c+\del))\leq\frac 1M
 \sum_{k=1}^M\rho(\hat V^{c,M}_{m,n,k},\Phi(\frac 1c+\del)).
 \]
 Hence,
 \begin{equation}\label{4.13}
 P\{\rho( \hat V^{c}_{m,n},\Phi(\frac 1c+\del))>2\ve\}\leq\frac 1M
 \sum_{k=1}^MP\{\rho(\hat V^{c,M}_{m,n,k},\Phi(\frac 1c+\del))>2\ve\}.
 \end{equation}
 
 Now, collecting the estimates (\ref{4.4})--(\ref{4.6}), (\ref{4.8}), 
 (\ref{4.9}) and (\ref{4.11})--(\ref{4.13}) we conclude that
 \begin{equation}\label{4.14}
 P\{\max_{0\leq j\leq n-b_c(n)}\rho(V^c_{j,n},\Phi(1/c))\geq (8D+4)\ve\}
 \leq d_5^{-1}n^{-d_5}
 \end{equation}
 for some $d_5>0$ depending on $\ve$ but not on $n$. Replacing $n$ by the
 subsequence $k_n=n^{2/d_5]}$ we obtain by the Borel-Cantelli lemma that with
 probability one,
 \begin{equation}\label{4.15}
 \limsup_{n\to\infty}\max_{0\leq j\leq k_n-b_c(k_n)}\rho(V^c_{j,k_n},\Phi(1/c))
 \leq (8D+4)\ve.
 \end{equation}
 Now take into account that if $k_n<r\leq k_{n+1}$ then
 \[
 b_c(k_{n+1})-b_c(r)\leq b_c(k_{n+1})-b_c(k_n)\leq c[2/d_4]\ln(\frac {n+1}n) +1
 \to 1\quad\mbox{as}\quad n\to\infty .
 \]
 It follows that (\ref{4.15}) remains true if $k_n$ there is replaced by $n$ 
 implying (\ref{4.1}) since $\ve>0$ is arbitrary.

  \section{The lower bound}\label{sec5}\setcounter{equation}{0}

We will prove here that with probability one
\begin{equation}\label{5.1}
\lim_{n\to\infty}\sup_{\gam\in\Phi(1/c)}\min_{0\leq j\leq n-b_c(n)}\rho
(V^c_{j,n},\gam)=0
\end{equation}
which will complete the proof of (\ref{2.8}). For $\gam\in\Phi(1/c)$ introduce
 the events
 \[
 \Gam_n^{(1)}=\Gam_n^{(1)}(\gam,\ve)=\{\min_{0\leq j\leq n-b_c(n)}
 \rho(V^c_{j,n},\gam)\geq 6\ve\}.
 \]
 Then
 \begin{equation}\label{5.2}
 \Gam_n^{(1)}\subset\Gam_n^{(2)}(\gam,\ve)=\{\min_{(1-1/4\ell)n\leq j\leq
 n-b_c(n)}\rho(V^c_{j,n},\gam)\geq 6\ve\}.
 \end{equation}
 Set
 \[
 \Gam^{(3)}_n(\gam,\ve)=\{\min_{(1-1/4\ell)n\leq j\leq
 n-b_c(n)}\rho(\hat V^c_{j,n},\gam)\geq 5\ve\}.
 \]
 Then, by (\ref{4.2}),
 \begin{equation}\label{5.3}
 P(\Gam^{(2)}_n(\gam,\ve))\leq P(\Gam^{(3)}(\gam,\ve))+d^{-1}_6e^{-d_6n}
 \end{equation}
 for some constant $d_6>0$ independent of $n$ and $\ve$.
 
 Next, we apply Lemma \ref{lem3.1} to random vectors
 \[
 Y_i=(X_{ij,[j/3]},,\ (1-\frac 1 {4\ell})n\leq j\leq n),\, i=1,...,\ell .
 \]
 Observe that
 \[
 (i-1/3)(1-1/4\ell)n-(i-2/3)n=n/3-(i-1/3)n/4\ell\geq n/12,
 \]
 and so similarly to Section \ref{sec4} we derive from (\ref{3.3}) that
 \begin{equation}\label{5.4}
 |P(\Gam_n^{(3)}(\gam,\ve))-P(\Gam_n^{(4)}(\gam,\ve))|\leq d_7^{-1}e^{-d_7n}
 \end{equation}
 where
 \[
 \Gam^{(4)}_n(\gam,\ve)=\{\min_{(1-1/4\ell)n\leq j\leq
 n-b_c(n)}\rho(\hat U^c_{j,n},\gam)\geq 5\ve\}
 \]
 and $d_7>0$ does not depend on $n$. Now,
 \begin{equation}\label{5.5}
 \Gam_n^{(4)}=\cap_{(1-1/4\ell)n\leq j\leq n-b_c(n)}\Gam_{n,j}^{(5)}\subset
 \cap_{j:(1-1/4\ell)n\leq j[4\ln^2n]\leq n-b_c(n)}\Gam_{n,j[4\ln^2n]}^{(5)}
 \end{equation}
 where
 \[
 \Gam^{(5)}_{n,j}(\gam,\ve)=\{\rho(\hat U^c_{j,n},\gam)\geq 5\ve\}.
 \]
 Next, we write
 \[
 \Gam^{(5)}_{n,j}(\gam,\ve)\subset\Gam^{(6)}_{n,j}(\gam,\ve)\cup\{\rho
 (U^c_{j,n},\hat U^c_{j,n})>\ve\}
 \]
 where $\Gam^{(6)}_{n,j}(\gam,\ve)=\{\rho(U^c_{j,n},\gam)\geq 4\ve\}$.
 Hence, by (\ref{4.8++}) and (\ref{5.2})--(\ref{5.5}),
 \begin{equation}\label{5.6}
 P(\Gam_n^{(1)}(\gam,\ve))\leq P\big(\cap_{j:\,(1-1/4\ell)n\leq j[4\ln^2n]\leq
 n-b_c(n)}\Gam^{(6)}_{n,j[4\ln^2n]}(\gam,\ve)\big)+d_8e^{-d_8n}
 \end{equation}
 for some $d_8>0$ independent of $n$ but depending on $\ve$.
 
 Next, we will use $\psi$-mixing and approximation properties of the product
 process $(X^{(1)}_j,X^{(2)}_{2j},...,X^{(\ell)}_{\ell j}),\, j=1,2,...$.
 Consider the product probability space $(\Om^\ell,\cF^\ell,P^\ell)=(\om,\cF,P)
 \times\cdots\times(\Om,\cF,P)$ ($\ell$-times product) and the $\sig$-algebras
 $\cF_{m_1,n_1;m_2,n_2;...;m_\ell,n_\ell}=\cF_{m_1,n_1}\times\cF_{m_2,n_2}
 \times\cdots\times\cF_{m_\ell,n_\ell}$ (where by this product we mean the
  minimal $\sig$-algebra containing all products of sets from the factors).
  Let
  \begin{equation}\label{5.7}
  \Gam\in\cF_{-\infty,m_1;-\infty,m_2;...;-\infty,m_\ell}\,\,\,\mbox{and}\,\,\,
  \Del\in\cF_{m_1+k_1,\infty;m_2+k_2,\infty;...;m_\ell+k_\ell,\infty}.
  \end{equation}
  We claim that
  \begin{equation}\label{5.8}
  |P^\ell(\Gam\cap\Del)-P^\ell(\Gam)P^\ell(\Del)|\leq P^\ell(\Gam)P^\ell(\Del)
  \sum_{i=1}^\ell\psi(k_i)\prod_{j=i+1}^\ell(1+\psi(k_j))
  \end{equation}
  where $\prod_{\ell+1}^\ell=1$ and $\psi$ is defined in (\ref{2.1}).
  
  Indeed, take first $\Gam=\Gam_1\times\Gam_2\times\cdots\times\Gam_\ell$
  and $\Del=\Del_1\times\Del_2\times\cdots\times\Del_\ell$ with $\Gam_i\in
  \cF_{-\infty,m_i}$  and $\Del_i\in\cF_{m_i+k_i,\infty},\, i=1,...,\ell$.
  We proceed by induction in $\ell$. For $\ell=1$ the inequality (\ref{5.8})
  follows from (\ref{2.1}). Now, suppose that (\ref{5.8}) holds true for
  $\ell-1$ in place of $\ell$. Then by (\ref{2.1}),
  \begin{eqnarray*}
  & |P^\ell(\Gam\cap\Del)-P^\ell(\Gam)P^\ell(\Del)|=|\prod_{i=1}^\ell P(\Gam_i
  \cap\Del_i)-\prod_{i=1}^\ell P(\Gam_i)P(\Del_i)|\\
  &\leq P(\Gam_\ell\cap\Del_\ell)||\prod_{i=1}^{\ell-1} P(\Gam_i
  \cap\Del_i)-\prod_{i=1}^{\ell-1}P(\Gam_i)P(\Del_i)|\\
  &+|P(\Gam_\ell\cap\Del_\ell)-P(\Gam_\ell)P(\Del_\ell)|\prod_{i=1}^{\ell-1}
  P(\Gam_i)P(\Del_i)\\
  &\leq (1+\psi(k_\ell))P(\Gam_\ell)P(\Del_\ell)|\prod_{i=1}^{\ell-1} P(\Gam_i
  \cap\Del_i)-\prod_{i=1}^{\ell-1}P(\Gam_i)P(\Del_i)|\\
  &+\psi(k_\ell)\prod_{i=1}^{\ell}P(\Gam_i)P(\Del_i)
  \end{eqnarray*}
  and we arrive at (\ref{5.8}) using the induction hypothesis taking into
  account that here
  \[
  P^\ell(\Gam)=\prod_{i=1}^\ell P(\Gam_i)\quad\mbox{and}\quad
  P^\ell(\Del)=\prod_{i=1}^\ell P(\Del_i).
  \]
  Since (\ref{5.8}) remains true under finite disjoint unions and under 
  monotone unions and intersections then by the monotone class theorem
  (\ref{5.8}) is valid for all $\Gam$ and $\Del$ satisfying (\ref{5.7}).
  
  Next, we redefine $X_j^{(i)}$ on the product space $\Om^\ell$ setting
  $X_j^{(i)}(\om,\om_2,...,\om_\ell)$ equal to $X_j^{(i)}(\om_i)$ and using
  the same notation for the new processes which are, again, independent for
  different $i$'s and have the same distributions as before. Using these
  $X^{(i)}_j$ we redefine on $\Om^\ell$ the sum $T_n$ and the random curves 
  $U^c_{m,n}(u),\, u\in[0,1]$, as before. Let $\cF_0=\{\emptyset,\Om\}$ be
  the trivial $\sig$-algebra on $\Om$ and set
  \[
  \cF^{(i)}=\cF_0\times\cdots\times\cF_0\times\cF\times\cF_0\times\cdots\times
  \cF_0\,\,\mbox{and}\,\,\cF^{(i)}=\cF_0\times\cdots\times\cF_0\times\cF_{m,n}
  \times\cF_0\times\cdots\times\cF_0
  \]
  where the nontrivial $\sig$-algebra appears as the $i$-th factor (and the
  product $\sig$-algebras are understood in the same sense as above). Set
  $X^{(i)}_{n,m}=E^\ell(X^{(i)}_n|\cF^{(i)}_{n-m,n+m})$ where $E^\ell$ is the
  expectation with respect to the probability $P^\ell$. Define
  \[
  \tilde T_n=\sum_{j=1}^nF(X^{(1)}_{j,[\ln^2n]},X^{(2)}_{2j,[\ln^2n]},...,
  X^{(\ell)}_{\ell j,[\ln^2n]})
  \]
  and
  \[
  \tilde U^c_{i,n}(u)=\frac {\tilde T_{j+b_c(n,u)}-\tilde T_j}{b_c(n)}.
  \]
  
  Taking into account that the pairs $(X^{(i)}_j,X^{(i)}_{j,[\ln^2n]})$
  have the same distributions as $(X_j,X_{j,[\ln^2n]})$ we derive from
  (\ref{2.3}) and (\ref{2.3+}) similarly to (\ref{4.8+}) and (\ref{4.8++})
  that
  \begin{equation}\label{5.9}
  P^\ell\{\rho(U^c_{j,n},\tilde U^c_{j,n})>\ve\}\leq d^{-1}_9\exp(-d_9\ln^2n)
  \end{equation}
  for some $d_9>0$ independent of $n$ but depending on $\ve$. Hence,
  \begin{eqnarray}\label{5.10}
  &P(\cap_{j:\, (1-1/4\ell)n\leq j[4\ln^2n]\leq n-b_c(n)}\Gam_{n,j[4\ln^2n]}
  ^{(6)}(\gam,\ve))\\
  &\leq P(\cap_{j:\, (1-1/4\ell)n\leq j[4\ln^2n]\leq n-b_c(n)}
  \Gam_{n,j[4\ln^2n]}^{(7)}
  (\gam,\ve))+d^{-1}_9n\exp(-d_9\ln^2n)\nonumber
  \end{eqnarray}
  where $\Gam^{(7)}_{n,j}(\gam,\ve)=\{\rho(\tilde U^c_{j,n},\gam)\geq 3\ve\}$.
  Now observe that $\Gam^{(7)}_{n,j[4\ln^2n]}(\gam,\ve)$ is $\cF_{k_1(j,n),
  m_1(j,n);k_2(j,n),m_2(j,n);...;k_\ell(j,n),m_\ell(j,n)}$-measurable, where
  $k_i(j,n)=ij[4\ln^2n]-[\ln^2n]$ and $m_i(j,n)=ij[4\ln^2n]+[\ln^2n]+b_c(n)$.
  Hence, $k_i(j+1,n)-m_i(j,n)\geq [\ln^2n]$ when $n$ is large enough. Applying
  (\ref{5.8}) successively we obtain that
  \begin{eqnarray}\label{5.11}
   &P(\cap_{j:\, (1-1/4\ell)n\leq j[4\ln^2n]\leq n-b_c(n)}\Gam_{n,j[4\ln^2n]}
   ^{(7)}(\gam,\ve))\\
  &\leq (1+\Psi([\ln^2n]))^{\frac n{[4\ln^2n]}}\prod_{j:\, 
  (1-1/4\ell)n\leq j[4\ln^2n]\leq n-b_c(n)}P^\ell(\Gam_{n,j[4\ln^2n]}
  ^{(7)}(\gam,\ve))\nonumber
  \end{eqnarray}
  where 
  \begin{equation}\label{5.12}
  \Psi(k)=\ell\psi(k)(1+\psi(k))^\ell\leq\ka_1^{-1}\ell 2^\ell e^{-\ka_1k}
  \end{equation}
  for all $k$ large enough.
  
  Next, we use (\ref{5.9}) again in order to obtain that
  \begin{equation}\label{5.13}
  P^\ell(\Gam^{(7)}_{n,j}(\gam,\ve))\leq P^\ell(\Gam^{(8)}_{n,j}(\gam,\ve))+
  d^{-1}_9\exp(-d_9\ln^2n)
  \end{equation}
  where $\Gam^{(8)}_{n,j}(\gam,\ve)=\{\rho(U^c_{j,n},\gam)\geq 
  2\ve\}$. Observe that we can replace $P^\ell$ by $P$ in the right hand side
  of (\ref{5.13}) if the independent processes $\{ X^{(i)}_{ij},\, j=1,2,...\},
  \, i=1,...,\ell$ are considered on the original probability space
  $(\Om,\cF,P)$.
  
  Since we consider $\gam\in\Phi(1/c)$ then $I(\dot\gam_u)<\infty$ for Lebesgue
  almost all $u\in[0,1]$ and (see Appendix) if $I(\be)<\infty,\,\be\in\bbR^d$
  then $I(a\be)<I(\be)$ for any $0<a<1$. Hence, if we define
  \[
  \eta_u=(1-\ve(\sup_{v\in[0,1]}|\gam_u|)^{-1})\gam_u,\, u\in[0,1]
  \]
  then
  \begin{equation}\label{5.14}
  \rho(\gam,\eta)\leq\ve\quad\mbox{and}\quad S(\eta)\leq S(\gam)-a\leq\frac 1c
  -a
  \end{equation}
  for some $a>0$. By (\ref{5.14}) and stationarity of the summands in $T_n$,
  \begin{eqnarray}\label{5.15}
  &P^\ell(\Gam^{(8)}_{n,j}(\gam,\ve))\leq P^\ell\{\rho(U^c_{j,n},\eta)\geq\ve\}
  \\
  &=P^\ell\{\rho(b^{-1}_c(n)T_{b_c(n)},\eta)\geq\ve\}=1-P^\ell\{\rho(U^c_{j,n},
  \eta)<\ve\}.\nonumber
  \end{eqnarray}
  Now, employing the lower large deviations bound from (\ref{2.7}) we obtain
  that for any $\ve,\la>0$ there exists $n_0>0$ such that for all $n\geq n_0$,
  \begin{equation}\label{5.16}
  P^\ell\{\rho(b^{-1}_c(n)T_{b_c(n)},\eta)<\ve\}\geq\exp(-b_c(n)(c^{-1}-a+\la))
  \geq n^{-1+c\sig}
  \end{equation}
  where we choose $\la>0$ so small that $\sig=a-\la>0$.
  
  Hence, we obtain by (\ref{5.6}), (\ref{5.10})--(\ref{5.13}), (\ref{5.15})
  and (\ref{5.16}) that for all $n$ large enough
  \begin{eqnarray}\label{5.17}
  &P(\Gam^{(1)}_n(\gam,\ve))\leq(1+\ka^{-1}\ell 2^\ell\exp(-\ka_1[\ln^2n]))^n
  (1-n^{-1+c\sig})^{\frac n{5\ell[4\ln^2n]}}\\
  &+d_8^{-1}e^{-d_8n}+2d^{-1}_9n\exp(-d_9\ln^2n)\leq d^{-1}_{10}
  \exp(-d_{10}\ln^2n)\nonumber
  \end{eqnarray}
  for some $d_{10}>0$ independent of $n$ but depending on $\ve$. Employing 
  the Borel-Cantelli lemma we conclude from (\ref{5.17}) that for any
  $\gam\in\Phi(1/c)$ with probability one
  \begin{equation}\label{5.18}
  \limsup_{n\to\infty}\min_{0\leq j\leq n-b_c(n)}\rho(V^c_{j,n},\gam)<6\ve.
  \end{equation}
  Since $\Phi(1/c)$ is a compact set we can choose there an $\ve$-net $\gam_1,
  \gam_2,...,\gam_{k(\ve)}$ and then with probability one (\ref{5.18}) will
  hold true simultaneously for all $\gam=\gam_i,\, i=1,...,k(\ve)$. It follows
  then that with probability one,
  \begin{equation}\label{5.19}
  \limsup_{n\to\infty}\sup_{\gam\in\Phi(1/c)}\min_{0\leq j\leq n-b_c(n)}\rho
  (V^c_{j,n},\gam)\leq 7\ve
  \end{equation}
  and since $\ve>0$ is arbitrary we obtain (\ref{5.1}).       
  \qed

 \section{Proof of Corollary \ref{cor2.2}}\label{sec6}
\setcounter{equation}{0}
Observe that (\ref{2.8}) implies, in particular, that for any continuous 
(with respect to the metric $\rho$) function $f$ on the space of curves 
$[0,1]\to\bbR^d$ with probability one,
\begin{equation}\label{6.1}
\lim_{n\to\infty}\max_{0\leq k\leq n-b_c(n)}f(V^c_{k,n})=
\sup_{\gam\in\Phi(1/c)}f(\gam).
\end{equation}
Now assume that $d=1$ and take $f(\gam)=\gam(1)$ where we write $\gam(u)=
\gam_u$. Assume that $I(\be)<\infty$ and set $c=1/I(\be)$. Then
\begin{equation}\label{6.2}
\sup_{\gam\in\Phi(I(\be))}f(\gam)=\sup\{\gam(1):\,\gam\in\Phi(I(\be))\}=\be.
\end{equation}

Indeed, by convexity of the rate function $I$ for any $\gam\in\Phi(I(\be))$,
\[
I(\be)\geq S(\gam)=\int_0^1I(\dot\gam(u))du\geq I(\int_0^1\dot\gam(u)du)
=I(\gam(1))
\]
and by monotonicity of $I$ (see Appendix), $\be\geq\gam(1)$. On the other
hand, take $\gam(u)=u\be$, $u\in[0,1]$. Then $S(\gam)=I(\be)$ and $\gam(1)=\be$
implying (\ref{6.2}) whenever $I(\be)<\infty$ and (\ref{2.9}) follows.  \qed

\section{Appendix}\label{sec7}\setcounter{equation}{0}
\subsection{Applications}\label{subsec7.1}
The main applications in the discrete time case of Theorem \ref{thm2.1}
concern Markov chains and some classes of dynamical systems such as Axiom
A diffeomorphisms, expanding transformations and topologically mixing 
subshifts of finite type. We will restrict ourselves to several main setups
to which our results are applicable rather than trying to describe most
general situations. First, let $X_n,\, n\geq 0$ be a time homogeneous
Markov chain on $\bbR^\wp$ whose transition probability
$P(x,\Gam)=P\{X_1\in\Gam|X_0=x\}$ satisfies
\begin{equation}\label{7.1}
\ka\nu(\Gam)\leq P(x,\Gam)\leq\ka^{-1}\nu(\Gam)
\end{equation}
for some $\ka>0$, a probability measure $\nu$ on $\bbR^\wp$ and any Borel set 
$\Gam\subset \bbR^\wp$. Then $X_n,\, n\geq 0$ is exponentially fast $\psi$-mixing
with respect to the family of $\sig$-algebras $\cF_{m,n}=\sig\{X_k,\,
m\leq k\leq n\}$ generated by the process (see, for instance, \cite{IL}).
The strong Doeblin type condition (\ref{7.1}) implies geometric ergodicity
\[
\| P(n,x,\cdot)-\mu\|\leq\be^{-1}e^{-\be n},\,\be>0
\]
where $\|\cdot\|$ is the variational norm, $P(n,x,\cdot)$ is the $n$-step 
transition probability and $\mu$ is the unique invariant measure of
$\{X_n,\, n\geq 0\}$ which makes it a stationary process. In this situation
$(X^{(1)}_n,X^{(2)}_{2n},...,X^{(\ell)}_{\ell n}),\, n\geq 0$ is the product 
Markov chain on $\bbR^{\ell\wp}$ satisfying similar to (\ref{7.1}) strong 
Doeblin condition. The limit (\ref{2.4}) exists here (see Lemma 4.3 in Ch.7 of
 \cite{FW}) and $\exp(\Pi(\al))$ turns out to be the principal eigenvalue
 of the positive operator 
 \[
 Gf(x)=E_xf(X_1^{(1)},X_2^{(2)},...,X_\ell^{(\ell)})\exp\big((\al,F(X_1^{(1)},
 X_2^{(2)},...,X_\ell^{(\ell)}))\big)
 \]
 (see \cite{Ki1} and references there) where $E_x$ is the expectation
 conditioned to $(X^{(1)}_0,X^{(2)}_0,...,X^{(\ell)}_0)=x$. It is well known 
 (see \cite{Na}, \cite{IL}, \cite{GH} and references there) that $\Pi(\al)$
 is convex and differentiable in $\al$. Furthermore, the Hessian matrix
 $\nabla_\al^2\Pi(\al)|_{\al=0}$ is positively definite if and only if for
 each $\al\in\bbR^d$, $\al\ne 0$ the limiting variance
 \begin{equation}\label{7.2}
 \sig_\al^2=\lim_{n\to\infty}n^{-1}E\big(\sum_{k=0}^n(\al,F(X^{(1)}_k,
 X^{(2)}_{2k},...,X^{(\ell)}_{\ell k}))\big)
 \end{equation}
 is positive. The latter holds true unless there exists a representation
 \[
 (\al,F(X^{(1)}_n,...,X^{(\ell)}_{\ell n}))=g(X^{(1)}_n,...,X^{(\ell)}_{\ell n})
 -g(X^{(1)}_{n-1},...,X^{(\ell)}_{\ell (n-1)}),\, n=1,2,...
 \]
 for some bounded Borel function $g$ (see \cite{IL}).

In the discrete time dynamical systems case we consider $X_n(\om)=g\circ f^n
(\om),\, n\geq 0$ 
where $g$ is a H\" older continuous vector function and
$f:\Om\to \Om$ is a $C^2$ Axion A diffeomorphism on a hyperbolic set or
a topologically mixing subshift of finite type or a $C^2$ expanding 
transformation. Here $X_n,\, n\geq 0$ is considered as a stationary
process on the probability space $(\Om,\cF,P)$ where $\Om$ is the corresponding
phase space, $\cF$ is the Borel
$\sig$-algebra and $P$ is a Gibbs measure constructed by a H\" older 
continuous function (see \cite{Bow}). Then the exponentially fast $\psi$-mixing
 holds true (see \cite{Bow}) with respect to the family of (finite)
$\sig$-algebras generated by cylinder sets in the symbolic setup of subshifts
of finite type or with respect to the corresponding $\sig$-algebras constructed
via Markov partitions in the Axiom A and expanding cases. 

Here the process $(X^{(1)}_n,...,X^{(\ell)}_{\ell n})$ is generated by the 
product dynamical system 
\[
(f\times f^2\times\cdots\times f^\ell)^n(\om_1,\om_2,...,\om_\ell)=
 (f^n\om_1,f^{2n}\om_2,...,f^{\ell n}\om_\ell)
 \]
 so that
 \begin{equation}\label{7.3}
 (X^{(1)}_n(\om_1),...,X^{(\ell)}_{\ell n}(\om_\ell))=
 G\circ(f\times f^2\times\cdots\times f^\ell)^n(\om_1,\om_2,...,\om_\ell)
 \end{equation}
 where $G(\om_1,\om_2,...,\om_\ell)=(g(\om_1),g(\om_2),...,g(\om_\ell))$.
 The above product dynamical system has similar properties as the original 
 dynamical system $f^n\om,\, n\geq 0$ itself, in particular, it satisfies
 large deviations bounds with respect to Gibbs measures constructed by
 H\" older continuous functions and exponentially fast $\psi$-mixing holds
 true, as well. The existence of the limit (\ref{2.4}) and its form
  follows from \cite{Ki2}. Here $\Pi_t(\al)$
turns out to be the topological pressure for the function $(\al, F)+\vf$
where $\vf$ is the potential of the corresponding Gibbs measure (for the
product system). The differentiability
properties of $\Pi_t(\al)$ in $\al$ are well known and, again, the Hessian
matrix $\nabla^2_\al\Pi_t(\al)|_{\al=0}$ is positively definite if and only if
for each $\al\in\bbR^d,\,\al\ne 0$ the limiting variance (\ref{7.2}) is
positive where the expectation should be taken with respect to the 
chosen Gibbs measure (see \cite{PP}, \cite{GH}, \cite{Ki1}, \cite{Ki2}
and references there). The latter holds true unless there exists a
coboundary representation $(\al,F)=g\circ f-g$ for some bounded Borel
function $g$.

\subsection{Some properties of rate functions}\label{subsec7.2}
We collect here few properties of rate functions of large deviations
which are essentially well known but hard to find in major books on large
deviations. First, observe that if $\Pi(\al)$, $\al\in\bbR^d$ is a twice
differentiable function such that $\Pi(0)=0,\,\nabla_\al\Pi(\al)|_{\al=0}
=0$ then $\Pi(\al)=o(|\al|)$, and so
\begin{equation}\label{7.4}
I(\be)=\sup_\al((\al,\be)-\Pi(\al))>0
\end{equation}
unless $\be=0$. Indeed, by the above
\[
I(\del\be)\geq\del|\be|^2-\Pi(\del\be)>0
\]
if $\del>0$ is small enough. Curiously, positivity of the rate function is not
studied in several books on large deviations without which upper large
deviations bounds do not make much sense.

Next, assume, in addition, that $\Pi$ is convex and has a positively definite
at zero Hessian matrix $\nabla_\al^2\Pi(\al)|_{\al=0}$. Then $\Pi(\al)\geq
0$ for all $\al\in\bbR^d$ and for some $\del_1,\del_2>0$,
\begin{equation}\label{7.5}
\Pi(\al)\geq\del_1|\al|\quad\mbox{provided}\quad |\al|>\del_2.
\end{equation}
It follows that if $|\be|<\del_1$ then $\al_\be=\arg\sup((\al,\be)-\Pi(\al))$ 
satisfies $|\al_\be|\leq\del_2$ and, in particular, $I(\be)<\infty$, i.e.
$I(\be)$ is finite in some neighborhood of $0$. 

Next, under the above conditions on $\Pi$ suppose that $I(\be)<\infty$
for some $\be\ne 0$. Then 
\begin{equation}\label{7.6}
I((1+\del)\be)>I(\be)\quad\mbox{for any}\quad\del>0.
\end{equation}
Indeed, for any $\ve>0$ there exists $\al_{\be,\ve}$ such that
\[
(\al_{\be,\ve},\be)-\Pi(\al_{\be,\ve})\geq I(\be)-\ve.
\]
Since $\Pi(\al_{\be,\ve})\geq 0$ we have
\[
I((1+\del)\be)\geq (1+\del)(\al_{\be,\ve},\be)-\Pi(\al_{\be,\ve})\geq
I(\be)+\del(I(\be)-\ve)-\ve>I(\be)
\]
provided $\ve< \del(1+\del)^{-1}I(\be)$ yielding (\ref{7.6}).

In the Erd\H os-R\' enyi law type results it is important to know where
a rate function $I(\be)$ is finite. This issue is hidden inside the functional
form of Theorems \ref{thm2.1} but appears explicitly in
Corollary \ref{cor2.2} and in its original form (\ref{1.1}). The discussion
on finiteness of rate functions is hard to find in books on large deviations
though without studying this issue lower bounds there do not have much sense.
We start with the rate functional $J(\nu)$ of the second level of large 
deviations for occupational measures
\begin{equation}\label{7.7}
\zeta_n=\frac 1n\sum_{k=0}^{n-1}\del_{X_k},
\end{equation}
 where $\del_x$ denotes
the unit mass at $x$ (see \cite{Ki1}). Explicit formulas for $J(\nu)$ are
known when $X_k$ is a Markov chain whose transition probability satisfies
(\ref{5.1}) and when $X_k=f^kx$ with $f$ being an Axiom A diffeomorphism,
expanding transformation or subshift of finite type. In the former case
(see \cite{DV}),
\begin{equation}\label{7.8}
J(\nu)=-\inf_{u>0,\,\mbox{\small continuous}}\int\ln(\frac {Pu}u)d\nu
\end{equation}
and in the latter case (see \cite{Ki1}),
\begin{equation}\label{7.9}
\gathered J(\nu)=\left\{
\aligned &-\int\varphi d\mu -h_\nu(f)
\ \text{if}\, \nu \, \text{is}\, f\text{-invariant,}\\
&\infty\ \text{otherwise}
\endaligned\right. \endgathered
\end{equation}
where $h_\nu(f)$ is the Kolmogorov--Sinai entropy of $f$ with respect to
 $\nu$ and $\vf$ is the potential of the corresponding Gibbs measure $\mu$
 playing the role of probability here.

Necessary and sufficient conditions for finiteness of $J(\nu)$ in the Markov
chain case are given in \cite{DV} while in the above
dynamical systems cases $J(\nu)<\infty$ for any $f$-invariant measure $\nu$.
If 
\begin{equation}\label{7.10}
\Pi(\al)=\lim_{n\to\infty}\frac 1n\ln E\exp\big(\sum_{j=0}^{n-1}(\al,
G(X_j))\big),
\end{equation}
 where $X_t$ is a stationary process as above and 
 $G\not\equiv 0$ is a continuous vector function
with $EG(X_0)=0$, then by the contraction principle (see, for instance,
\cite{DZ}) the rate function $I(\be)$ given by (\ref{7.4}) can be represented
 as
 \begin{equation}\label{7.11}
 I(\be)=\inf\{ J(\nu):\,\int Gd\nu=\be\}
 \end{equation}
 where the infinum is taken over the space $\cP(M)$ of probability measures
 on $M$.
 
 Set 
 \[
 \Gam=\{\be\in\bbR^d:\,\exists\nu\in\cP(M)\,\,\mbox{such that}\,\,\int Gd\nu=
 \be \,\,\mbox{and}\,\, J(\nu)<\infty\}
 \]
 and let $Co(\Gam)$ be the interior of the convex hull of $\Gam$. Then
 \begin{equation}\label{7.12}
 I(\be)<\infty\,\,\mbox{for any}\,\,\be\in Co(\Gam).
 \end{equation}
 Indeed, any $\be\in Co(\Gam)$ can be represented as $\be=p_1\be_1+p_2\be_2$
 with $\be_1,\be_2\in\Gam$, $p_1,p_2\geq 0$ and $p_1+p_2=1$. Then $\be_1=
 \int Gd\nu_1,\,\be_2=\int Gd\nu$, and so $\int Gd\nu=\be$ for $\nu=p_1\nu_1
 +p_2\nu_2$. Since $J(\nu_1),\, J(\nu_2)<\infty$ then by convexity of $J$ we
 have that $J(\nu)\leq p_1J(\nu_1)+p_2J(\nu_2)<\infty$, and so (\ref{7.12})
 holds true.
 
 When $d=1$, i.e. when $G$ is a (not vector) function we can give another
 description of the domain where $I(\be)<\infty$. In this case set
 \begin{equation}\label{7.13}
 \be_+=\sup\{\be:\,\be\in\Gam\}\,\,\mbox{and}\,\,\be_-=\inf\{\be:\be\in\Gam\}.
 \end{equation}
 Then by (\ref{7.12}), $I(\be)<\infty$ for any $\be\in (\be_-,\be_+)$. It is
 possible to extract from \cite{DK} that under $\psi$-mixing,
 \begin{equation}\label{7.14}
 \be_+=\lim_{n\to\infty}\frac 1n ess\sup\sum_{j=0}^{n-1}G(X_j)\,\,\mbox{and}
 \,\, \be_-=\lim_{n\to\infty}\frac 1ness\inf\sum_{j=0}^{n-1}G(X_j).
 \end{equation}

\bibliography{matz_nonarticles,matz_articles}
\bibliographystyle{alpha}

\end{document}